\documentclass[11pt]{article}
\usepackage{amssymb}
\usepackage{amsmath}
\usepackage{amsfonts}
\usepackage{amsthm}
\usepackage{amsmath}

\usepackage{tikz}

\usepackage[english]{babel}
\usepackage[affil-it]{authblk}

  \makeatletter
\def\@fnsymbol#1{\ensuremath{\ifcase#1\or  \ddagger\or
		\mathsection\or \mathparagraph\or \|\or **\or \dagger\dagger
		\or \ddagger\ddagger \else\@ctrerr\fi}}
\makeatother

\newtheorem{theorem}{Theorem}

\newtheorem{problem}{Problem}
\newtheorem{proposition}{Proposition}
\newtheorem{corollary}{Corollary}
\newtheorem{example}{Example} 

\newtheorem{construction}{Construction}

\providecommand{\keywords}[1]{\textbf{\textit{Keywords:}} #1}

\title{Wiener index and graphs, almost half of whose vertices satisfy \v{S}olt\'{e}s property}

\author[1]{Margarita Akhmejanova}
\author[1]{Konstantin Olmezov}
\author[1]{Aleksei Volostnov}
\author[2]{ Ilya Vorobyev}
\author[3]{Konstantin Vorob'ev}
\author[1]{Yury Yarovikov}

\affil[1]{Moscow Institute of Physics and Technology, Dolgoprudny, Russia. E-mail: mechmathrita@gmail.com, olmezov.ki@gmail.com, gyololo@rambler.ru, yu-rovikov@yandex.ru}
\affil[2]{Skolkovo Insitute of Science and Technology, Russia. E-mail: vorobyev.i.v@yandex.ru}
\affil[3]{Sobolev Institute of Mathematics, Novosibirsk, Russia. E-mail: vorobev@math.nsc.ru}

\begin{document}

\maketitle 

\begin{abstract}
The Wiener index $W(G)$ of a connected graph $G$ is a sum of distances between all pairs of vertices of $G$. 
In 1991, \v{S}olt\'{e}s formulated the problem of finding all graphs $G$ such that for every vertex $v$ the equality $W(G)=W(G-v)$ holds. The cycle $C_{11}$ is the only known graph with this property. In this paper we consider the following relaxation of the original problem: find a graph with a large proportion of vertices such that removing any one of them does not change the Wiener index of a graph.  
As the main result, we build an infinite series of graphs with the proportion of such vertices tending to $\frac{1}{2}$.

\end{abstract}


\section{Introduction}\label{S:Introduction}

Let $G=(V,E)$ be a simple connected graph. The {\it Wiener index} of a graph $G$ is defined as the sum of distances between all pairs of vertices of $G$:
$$W(G)=\frac{1}{2}\sum_{v,u\in V, v\neq u}{d_G(u,v)},$$
where $d_G(x,y)$ is a usual graph distance, i.e. the length of a shortest path from $x$ to $y$ in $G$.
This index was introduced in $1947$ by Wiener \cite{W1947} for applications in chemistry, where he used similar expression to determine the boiling point of the paraffin. In $1971$, Hosoya \cite{H1971} first defined this index in terms of graph theory. After that, the index for different families of graphs was investigated because of pure mathematical interest. Very recently, Egorov and Vesnin \cite{EgorVesn} investigated correlation of hyperbolic volumes of fullerenes (fullerene is a molecule consisting entirely of carbon atoms) with Wiener index and other related indexes. We also refer an interested reader to survey papers by Dobrynin, Entringer and Gutman \cite{DEG2001} and  Knor, \v{S}krekovski and Tepeh \cite{KST2016}.

Let us define the {\it transmission} of a vertex $v\in V$ in $G$ as $$t_G(v)=\sum_{u\in V\setminus{\{v\}}}{d_G(v,u)}.$$
Clearly, the Wiener index of a graph $G$ may be expressed in terms of vertices transmissions: 
\begin{equation*}
W(G)=\frac{1}{2}\sum_{v\in V}{t_G(v)}.
\end{equation*}
Let us denote by $G-v$ the graph obtained by removing the vertex $v$ from $G$. A vertex $v\in V$ is called {\it good} if $G-v$ is connected and $W(G) = W(G - v)$. In 1991, \v{S}olt\'{e}s \cite{S1991} posed the following problem.\\ 

\begin{problem}\label{Prob:1}
Find all graphs $G$ such that the equality $W(G) = W(G - v)$ holds for all their vertices $v$.
\end{problem}
Today the only known graph with this property is $C_{11}$. In other words, \v{S}olt\'{e}s asked about graphs such that all their vertices are good. After thirty years of study of this problem, which led to lots of elegant ideas, the problem is far from being solved. However, there are some partial results. In $2018$, Knor, Majstorovi\'{c} and \v{S}krekovski in \cite{KMS2018dam} constructed an infinite series of graphs with one good vertex. The same authors also showed \cite{KMS2018amc} that for $k\geq 3$ there are infinitely many graphs with one good vertex of degree $k$. In \cite{BJM2019}, for any $k\in \mathbb{N}$, Bok, Jedli\v{c}kov\'{a} and Maxov\'{a} found an infinite series of graphs with exactly $k$ good vertices. 

For a graph $G=(V,E)$ and vertices $v,u,w\in V$, let us define functions $\Delta_v(G)=W(G)-W(G-v)$ and $\delta_G^v(u,w)=d_G(u,w)-d_{G-v}(u,w)$. From the definitions we directly have the following 

\begin{proposition}\label{Prop:1}
	Given a graph $G=(V,E)$, $v\in V$. If $G-v$ is connected then the following equality holds:
	$$\Delta_v(G)=t_G(v)+\frac{1}{2}\sum_{u,w \in V\setminus \{v\}}{\delta^v_G(u,w)}.$$
	
\end{proposition} 
Instead of finding graphs with a fixed number of good vertices, one may try to find graphs with a fixed proportion of such vertices. So we denote a {\it proportion of good vertices} in $G$ as $\frac{|\{v\in V|\Delta_v(G)=0\}|}{|V|}.$ Now we are ready to formulate the following relaxation of Problem \ref{Prob:1}. 

\begin{problem}\label{Prob:2} 
For a fixed $\alpha\in (0,1]$ construct an infinite series $S$ of graphs such that for all $G=(V,E)$ from $S$ the following inequality takes place: $$ \frac{|\{v\in V|\Delta_v(G)=0\}|}{|V|} \geq \alpha.$$
\end{problem}

Clearly, a solution to this problem for $\alpha=1$ would give an infinite series of solutions to the Problem \ref{Prob:1}.
Let us note that the construction from \cite{BJM2019} after a slight modification yields an infinite series of graphs with the proportion of good vertices tending to $\frac{1}{3}$ (we will discuss it in more details in Section \ref{S:Lily}).

In this work, we provide two constructions  based on different ideas in order to improve this constant. 
In Section 2, we find an infinite series of graphs with the proportion of good vertices tending to $\frac{2}{5}$ as the number of vertices tends to infinity. In Section 3, we find one more series with the proportion tending to $\frac{1}{2}$. In Section 4, we discuss results and open problems.

\section{Bunch of $11$-cycles}\label{S:Bunch}
As it was mentioned in Section \ref{S:Introduction}, $C_{11}$ is the only one known graph where all vertices are good. Now we provide an intriguing construction of graphs with many $C_{11}$ as induced subgraphs.  

\begin{construction}\label{C:bunch} Given $k\in \mathbb{N}$, $k>1$. Let us define a graph $B(k)$.
	\begin{itemize}
		\item Step $1$. Take a path of length $5$ and denote its consecutive vertices by $w_0$, $w_1$, $w_2$, $w_3$, $w_4$ and $w_5$.
		\item Step $2$. Take $k$ distinct paths $(v_{i1},v_{i2},v_{i3},v_{i4},v_{i5})$, $i\in \{1,2, \dots, k\}$. 
		\item Step $3$. For $i\in \{1,2, \dots, k\}$, add edges $(w_0, v_{i1})$ and $(w_5, v_{i5})$ to our graph.
	\end{itemize}
\end{construction}

\begin{center}
\begin{tikzpicture}[scale=1, every node/.style={scale=0.9}, rotate=270, yscale=-1, xscale=-1 ]


\node[draw, circle] (w0) at (0,1) {$w_0$};
\node[draw, circle] (w5) at (0,-7) {$w_5$};


\node[draw, circle] (w1) at (-4,-1) {$w_1$};
\node[draw, circle] (w2) at (-4,-2.33) {$w_2$};
\node[draw, circle] (w3) at (-4,-3.66) {$w_3$};
\node[draw, circle] (w4) at (-4,-5) {$w_4$};

\draw [thick]  (w0) to (w1);
\draw [thick]  (w1) to (w2);
\draw [thick]  (w2) to (w3);
\draw [thick]  (w3) to (w4);
\draw [thick]  (w4) to (w5);


\node[draw, circle] (v11) at (-2,-1) {$v_{11}$};
\node[draw, circle] (v12) at (-2,-2) {$v_{12}$};
\node[draw, circle] (v13) at (-2,-3) {$v_{13}$};
\node[draw, circle] (v14) at (-2,-4) {$v_{14}$};
\node[draw, circle] (v15) at (-2,-5) {$v_{15}$};

\draw [thick]  (w0) to (v11);
\draw [thick]  (v11) to (v12);
\draw [thick]  (v12) to (v13);
\draw [thick]  (v13) to (v14);
\draw [thick]  (v14) to (v15);
\draw [thick]  (v15) to (w5);

\node[draw, circle] (v21) at (-1,-1) {$v_{21}$};
\node[draw, circle] (v22) at (-1,-2) {$v_{22}$};
\node[draw, circle] (v23) at (-1,-3) {$v_{23}$};
\node[draw, circle] (v24) at (-1,-4) {$v_{24}$};
\node[draw, circle] (v25) at (-1,-5) {$v_{25}$};

\draw [thick]  (w0) to (v21);
\draw [thick]  (v21) to (v22);
\draw [thick]  (v22) to (v23);
\draw [thick]  (v23) to (v24);
\draw [thick]  (v24) to (v25);
\draw [thick]  (v25) to (w5);

\draw [fill=black!100, color=black!100] (-0.2,-3) circle (0.01cm);
\draw [fill=black!100, color=black!100] (0,-3) circle (0.01cm);
\draw [fill=black!100, color=black!100] (0.2,-3) circle (0.01cm);

\node[draw, circle] (vk1) at (1,-1) {$v_{k1}$};
\node[draw, circle] (vk2) at (1,-2) {$v_{k2}$};
\node[draw, circle] (vk3) at (1,-3) {$v_{k3}$};
\node[draw, circle] (vk4) at (1,-4) {$v_{k4}$};
\node[draw, circle] (vk5) at (1,-5) {$v_{k5}$};

\draw [thick]  (w0) to (vk1);
\draw [thick]  (vk1) to (vk2);
\draw [thick]  (vk2) to (vk3);
\draw [thick]  (vk3) to (vk4);
\draw [thick]  (vk4) to (vk5);
\draw [thick]  (vk5) to (w5);

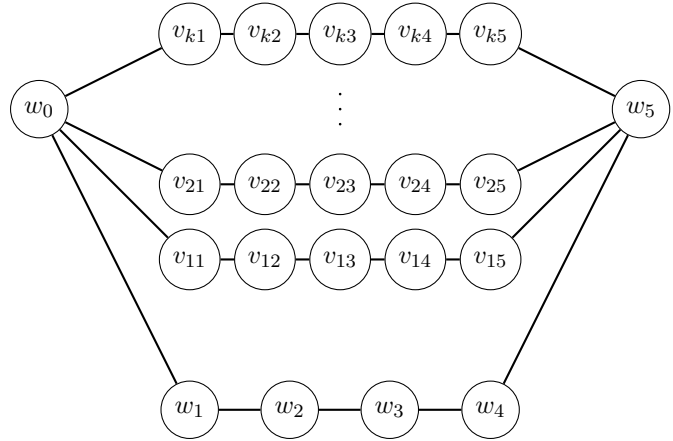
\node[below,font=\large\bfseries] at (current bounding box.south) {Figure 1. Graph $B(k)$ on $5k+6$ vertices};

\end{tikzpicture}
\end{center}

\begin{theorem}\label{T:bunch}
	For $k\in \mathbb{N}$, $k\geq 2$, the proportion of good vertices in $B(k)$ equals $\frac{2k}{5k+6}$.
\end{theorem}
\begin{proof}
	Consider a vertex $v=v_{11}$ of the graph $G=B(k)=(V,E)$. Our goal is to show that $\Delta_{v}(B(k))=0$. According to Proposition \ref{Prop:1}, one need to find $t_G(v)$ and $\delta^{v}_G(u,w)$ for all vertices $u,w\in V\setminus \{v\}$.
	By direct calculations we have that $$t_G(v)=1+2+3+4+5+5+4+3+2+1+(k-1)(2+3+4+5+6)=20k+10.$$ It is easy to see from the construction of the graph that $\delta^{v}_G(u,w)=0$ whenever $u,w\notin \{v_{12},v_{13},v_{14},v_{15}\}$.  All nonzero values of $\delta^{v}_G(u,w)$ are presented in the Table \ref{Tab:bunch} (it also contains some zeros):

		\begin{table}[h]
		
		\centering
		\begin{tabular}{ |c|c|c|c|c|c|c|c|}
			
			\hline
			Vertices & $w_0$ & $w_1$ & $w_2$ & $w_3$ & $v_{i1}$ & $v_{i2}$ & $v_{i3}$ \\
			\hline

			$v_{12}$ & $-7$ & $-5$ & $-3$ & $-1$ & $-6$ & $-4$ & $-2$ \\
			$v_{13}$ & $-5$ & $-3$ & $-1$ & $0$ & $-4$ & $-2$ & $0$ \\
			$v_{14}$ & $-3$ & $-1$ & $0$ & $0$ & $-2$ & $0$ & $0$ \\
			$v_{15}$ & $-1$ & $0$ & $0$ & $0$ & $0$ & $0$ & $0$\\
		
			\hline
		\end{tabular}
		\caption{Values of $\delta^{v}_G(u,w)$, $i\in \{2,3,\dots, k\}$}
		{\label{Tab:bunch}}
	\end{table}	
Consequently, $$\frac{1}{2}\sum_{u,w \in V\setminus \{v\}}{\delta^{v}_G(u,w)}=-(7+5+3+1+5+3+1+3+1+1)-(k-1)(6+4+2+4+2+2)=-20k-10.$$
	By Proposition \ref{Prop:1} we have that $\Delta_{v}(G)=0$.
	Based on the symmetry of $B(k)$ we conclude that for all $z\in \{v_{11}, v_{21}, \dots, v_{k1}, v_{15}, v_{25}, \dots, v_{k5}\}$ we have $\Delta_{z}(G)=0$. By similar calculations it is easy to check, that for other vertices $z\in V$ we have $\Delta_{z}(G)\neq 0$. Therefore, $B(k)$ has exactly $2k$ good vertices among its $5k+6$ ones.

\end{proof}

\begin{corollary}\label{Cor:bunch}
	The series of graphs $B(k)$, $k\in \mathbb{N}$, has a proportion of good vertices tending to $\frac{2}{5}$ as $k$ tends to infinity.
\end{corollary}

From Construction \ref{C:bunch} one may see that the graph $B(k)$ is $2$-connected, so the value $\Delta_v(B(k))$ is defined for any vertex $v$. In the next Section, we let a vertex-connectivity be $1$. It allows us to present one more construction with a proportion of good vertices tending to $\frac{1}{2}$.   

\section{Lily-shaped construction}\label{S:Lily}

 In this Section, we find one more series of graphs with a proportion of good vertices tending to $\frac{1}{2}$.
\begin{construction}\label{C:lily} Given $k,m\in \mathbb{N}$, $m\geq 7$ and $k\geq \frac{m-3}{m-6}$. Let us define a graph $L(k,m)$.
	\begin{itemize}
		\item Step $1$. 
		Define a graph containing $k$ distinct paths of length $3$ of the form $(u_0, v^i_{1j}, v^i_{2j}, v_i)$, $j\in \{1,2,\dots, k\}$, $i=1$. Note, that all these paths have two common vertices $u_0$ and $v_i$. In the rest of the paper, this graph will be called a {\it block}. 
		\item Step $2$. Add  $m-1$ distinct blocks of the same form to our graph. The only difference is a changing index $i\in \{2,3,\dots, m\}$. Note, that all these blocks contain exactly one common vertex $u_0$.
		\item Step $3$. For all $i_1,i_2\in \{1,2,\dots m\}$ and $j_1,j_2\in \{1,2,\dots k\}$ such that $i_1 \neq i_2$, add the edge $(v^{i_1}_{1j_1},v^{i_2}_{1j_2})$ to our graph. In other words, the graph induced by the set of vertices $\{v^i_{1j}| i\in \{1,2,\dots m\}, j\in \{1,2,\dots k\}\}$ is a the complete $m$-partite graph with parts $\{v^i_{11}, v^i_{12}, \dots , v^i_{1k} \}$, $i\in \{1,2,\dots, m\}$.
		\item Step $4$. Add a path $(u_0, u_1, u_2, \dots, u_d)$ to our graph, where $$d=\Big\lfloor{\frac{\sqrt{8km-48k-8m+33}+1}{2}}\Big\rfloor.$$ 
		\item Step $5$. Add one neighbour $u'_{t}$ to the vertex $u_t$, where $$t=km-6k-m+3-\frac{(d+1)(d-2)}{2}.$$
	\end{itemize}
\end{construction}

\begin{center}

\begin{tikzpicture}[thick, scale=1, every node/.style={scale=0.9}]

\def \l {1} 

\def \d {3.3}

\def \alfabf {30}

\node[draw, rectangle] (w) at (0, 0) {$u_0$};
\node[draw, rectangle] (b1v1) at ({3*\d*cos(\alfabf)}, {3*\d*sin(\alfabf))}) {$v_1$};

\def \rr {(\d^2+(3*\l/2)^2)^(1/2)}
\def \alfaa {atan((3/2*\l)/(\d))}
\node[draw, rectangle] (b1v11) at ({\rr*cos(\alfabf+\alfaa)} , {\rr*sin(\alfabf+\alfaa)}) {$v^1_{11}$};

\def \rrr {((2*\d)^2+(3*\l/2)^2)^(1/2)}
\def \alfaaa {atan((3/2*\l)/(2*\d))}
\node[draw, rectangle] (b1v21) at ({\rrr*cos(\alfabf+\alfaaa)} , {\rrr*sin(\alfabf+\alfaaa)}) {$v^1_{21}$};

\def \rrrr {(\d^2+(\l/2)^2)^(1/2)}
\def \alfaaaa {atan((\l/2)/(\d))}
\node[draw, rectangle] (b1v12) at ({\rrrr*cos(\alfabf+\alfaaaa)} , {\rrrr*sin(\alfabf+\alfaaaa)}) {$v^1_{12}$};

\def \rrrrr {((2*\d)^2+(\l/2)^2)^(1/2)}
\def \alfaaaaa {atan((\l/2)/(2*\d))}
\node[draw, rectangle] (b1v22) at ({\rrrrr*cos(\alfabf+\alfaaaaa)} , {\rrrrr*sin(\alfabf+\alfaaaaa)}) {$v^1_{22}$};

\def \rd {((3*\d/2)^2+(\l/2)^2)^(1/2)}
\def \alfad {atan((\l/2)/(3*\d/2))}
\draw [fill=black!100, color=black!100] ({\rd*cos(\alfabf-\alfad)} , {\rd*sin(\alfabf-\alfad)}) circle (0.01cm);

\def \rdl {((3*\d/2)^2+(\l/4)^2)^(1/2)}
\def \alfadl {atan((\l/4)/(3*\d/2))}
\draw [fill=black!100, color=black!100] ({\rdl*cos(\alfabf-\alfadl)} , {\rdl*sin(\alfabf-\alfadl)}) circle (0.01cm);

\def \rdr {((3*\d/2)^2+(3*\l/4)^2)^(1/2)}
\def \alfadr {atan((3*\l/4)/(3*\d/2))}
\draw [fill=black!100, color=black!100] ({\rdr*cos(\alfabf-\alfadr)} , {\rdr*sin(\alfabf-\alfadr)}) circle (0.01cm);

\def \rrrrrr {(\d^2+(3*\l/2)^2)^(1/2)}
\def \alfaaaaaa {atan((3*\l/2)/(\d))}
\node[draw, rectangle] (b1v1k) at ({\rrrrrr*cos(\alfabf-\alfaaaaaa)} , {\rrrrrr*sin(\alfabf-\alfaaaaaa)}) {$v^1_{1k}$};

\def \rrrrrrr {((2*\d)^2+(3*\l/2)^2)^(1/2)}
\def \alfaaaaaaa {atan((3*\l/2)/(2*\d))}
\node[draw, rectangle] (b1v2k) at ({\rrrrrrr*cos(\alfabf-\alfaaaaaaa)} , {\rrrrrrr*sin(\alfabf-\alfaaaaaaa)}) {$v^1_{2k}$};

\draw [thick]  (w) to (b1v11);
\draw [thick]  (w) to (b1v12);
\draw [thick]  (w) to (b1v1k);

\draw [thick]  (b1v11) to (b1v21);
\draw [thick]  (b1v12) to (b1v22);
\draw [thick]  (b1v1k) to (b1v2k);

\draw [thick]  (b1v21) to (b1v1);
\draw [thick]  (b1v22) to (b1v1);
\draw [thick]  (b1v2k) to (b1v1);


\def \alfabs {150}

\node[draw, rectangle] (b2v1) at ({3*\d*cos(\alfabs)}, {3*\d*sin(\alfabs))}) {$v_2$};

\def \rr {(\d^2+(3*\l/2)^2)^(1/2)}
\def \alfaa {atan((3/2*\l)/(\d))}
\node[draw, rectangle] (b2v11) at ({\rr*cos(\alfabs+\alfaa)} , {\rr*sin(\alfabs+\alfaa)}) {$v^2_{11}$};

\def \rrr {((2*\d)^2+(3*\l/2)^2)^(1/2)}
\def \alfaaa {atan((3/2*\l)/(2*\d))}
\node[draw, rectangle] (b2v21) at ({\rrr*cos(\alfabs+\alfaaa)} , {\rrr*sin(\alfabs+\alfaaa)}) {$v^2_{21}$};

\def \rrrr {(\d^2+(\l/2)^2)^(1/2)}
\def \alfaaaa {atan((\l/2)/(\d))}
\node[draw, rectangle] (b2v12) at ({\rrrr*cos(\alfabs+\alfaaaa)} , {\rrrr*sin(\alfabs+\alfaaaa)}) {$v^2_{12}$};

\def \rrrrr {((2*\d)^2+(\l/2)^2)^(1/2)}
\def \alfaaaaa {atan((\l/2)/(2*\d))}
\node[draw, rectangle] (b2v22) at ({\rrrrr*cos(\alfabs+\alfaaaaa)} , {\rrrrr*sin(\alfabs+\alfaaaaa)}) {$v^2_{22}$};

\def \rd {((3*\d/2)^2+(\l/2)^2)^(1/2)}
\def \alfad {atan((\l/2)/(3*\d/2))}
\draw [fill=black!100, color=black!100] ({\rd*cos(\alfabs-\alfad)} , {\rd*sin(\alfabs-\alfad)}) circle (0.01cm);

\def \rdl {((3*\d/2)^2+(\l/4)^2)^(1/2)}
\def \alfadl {atan((\l/4)/(3*\d/2))}
\draw [fill=black!100, color=black!100] ({\rdl*cos(\alfabs-\alfadl)} , {\rdl*sin(\alfabs-\alfadl)}) circle (0.01cm);

\def \rdr {((3*\d/2)^2+(3*\l/4)^2)^(1/2)}
\def \alfadr {atan((3*\l/4)/(3*\d/2))}
\draw [fill=black!100, color=black!100] ({\rdr*cos(\alfabs-\alfadr)} , {\rdr*sin(\alfabs-\alfadr)}) circle (0.01cm);

\def \rrrrrr {(\d^2+(3*\l/2)^2)^(1/2)}
\def \alfaaaaaa {atan((3*\l/2)/(\d))}
\node[draw, rectangle] (b2v1k) at ({\rrrrrr*cos(\alfabs-\alfaaaaaa)} , {\rrrrrr*sin(\alfabs-\alfaaaaaa)}) {$v^2_{1k}$};

\def \rrrrrrr {((2*\d)^2+(3*\l/2)^2)^(1/2)}
\def \alfaaaaaaa {atan((3*\l/2)/(2*\d))}
\node[draw, rectangle] (b2v2k) at ({\rrrrrrr*cos(\alfabs-\alfaaaaaaa)} , {\rrrrrrr*sin(\alfabs-\alfaaaaaaa)}) {$v^2_{2k}$};

\draw [thick]  (w) to (b2v11);
\draw [thick]  (w) to (b2v12);
\draw [thick]  (w) to (b2v1k);

\draw [thick]  (b2v11) to (b2v21);
\draw [thick]  (b2v12) to (b2v22);
\draw [thick]  (b2v1k) to (b2v2k);

\draw [thick]  (b2v21) to (b2v1);
\draw [thick]  (b2v22) to (b2v1);
\draw [thick]  (b2v2k) to (b2v1);


\def \alfabt {270}

\node[draw, rectangle] (bmv1) at ({3*\d*cos(\alfabt)}, {3*\d*sin(\alfabt))}) {$v_m$};

\def \rr {(\d^2+(3*\l/2)^2)^(1/2)}
\def \alfaa {atan((3/2*\l)/(\d))}
\node[draw, rectangle] (bmv11) at ({\rr*cos(\alfabt+\alfaa)} , {\rr*sin(\alfabt+\alfaa)}) {$v^m_{11}$};

\def \rrr {((2*\d)^2+(3*\l/2)^2)^(1/2)}
\def \alfaaa {atan((3/2*\l)/(2*\d))}
\node[draw, rectangle] (bmv21) at ({\rrr*cos(\alfabt+\alfaaa)} , {\rrr*sin(\alfabt+\alfaaa)}) {$v^m_{21}$};

\def \rrrr {(\d^2+(\l/2)^2)^(1/2)}
\def \alfaaaa {atan((\l/2)/(\d))}
\node[draw, rectangle] (bmv12) at ({\rrrr*cos(\alfabt+\alfaaaa)} , {\rrrr*sin(\alfabt+\alfaaaa)}) {$v^m_{12}$};

\def \rrrrr {((2*\d)^2+(\l/2)^2)^(1/2)}
\def \alfaaaaa {atan((\l/2)/(2*\d))}
\node[draw, rectangle] (bmv22) at ({\rrrrr*cos(\alfabt+\alfaaaaa)} , {\rrrrr*sin(\alfabt+\alfaaaaa)}) {$v^m_{22}$};

\def \rd {((3*\d/2)^2+(\l/2)^2)^(1/2)}
\def \alfad {atan((\l/2)/(3*\d/2))}
\draw [fill=black!100, color=black!100] ({\rd*cos(\alfabt-\alfad)} , {\rd*sin(\alfabt-\alfad)}) circle (0.01cm);

\def \rdl {((3*\d/2)^2+(\l/4)^2)^(1/2)}
\def \alfadl {atan((\l/4)/(3*\d/2))}
\draw [fill=black!100, color=black!100] ({\rdl*cos(\alfabt-\alfadl)} , {\rdl*sin(\alfabt-\alfadl)}) circle (0.01cm);

\def \rdr {((3*\d/2)^2+(3*\l/4)^2)^(1/2)}
\def \alfadr {atan((3*\l/4)/(3*\d/2))}
\draw [fill=black!100, color=black!100] ({\rdr*cos(\alfabt-\alfadr)} , {\rdr*sin(\alfabt-\alfadr)}) circle (0.01cm);

\def \rrrrrr {(\d^2+(3*\l/2)^2)^(1/2)}
\def \alfaaaaaa {atan((3*\l/2)/(\d))}
\node[draw, rectangle] (bmv1k) at ({\rrrrrr*cos(\alfabt-\alfaaaaaa)} , {\rrrrrr*sin(\alfabt-\alfaaaaaa)}) {$v^m_{1k}$};

\def \rrrrrrr {((2*\d)^2+(3*\l/2)^2)^(1/2)}
\def \alfaaaaaaa {atan((3*\l/2)/(2*\d))}
\node[draw, rectangle] (bmv2k) at ({\rrrrrrr*cos(\alfabt-\alfaaaaaaa)} , {\rrrrrrr*sin(\alfabt-\alfaaaaaaa)}) {$v^m_{2k}$};

\draw [thick]  (w) to (bmv11);
\draw [thick]  (w) to (bmv12);
\draw [thick]  (w) to (bmv1k);

\draw [thick]  (bmv11) to (bmv21);
\draw [thick]  (bmv12) to (bmv22);
\draw [thick]  (bmv1k) to (bmv2k);

\draw [thick]  (bmv21) to (bmv1);
\draw [thick]  (bmv22) to (bmv1);
\draw [thick]  (bmv2k) to (bmv1);


\def \rdd {((3*\d/2)}

\def \alfadd {(\alfabt+\alfabs)/2}
\draw [fill=black!100, color=black!100] ({\rdd*cos(\alfadd)} , {\rdd*sin(\alfadd)}) circle (0.01cm);

\def \alfaddl {((\alfabt+\alfabs)/2*3+(\alfabs+atan((3*\l/2)/\d)))/4}
\def \alfaddr {((\alfabt+\alfabs)/2*3+(\alfabt-atan((3*\l/2)/\d)))/4}
\draw [fill=black!100, color=black!100] ({\rdd*cos(\alfaddl)} , {\rdd*sin(\alfaddl)}) circle (0.01cm);
\draw [fill=black!100, color=black!100] ({\rdd*cos(\alfaddr)} , {\rdd*sin(\alfaddr)}) circle (0.01cm);

\def \xstep {0}
\def \ystep {-\d/3.5}
\def \xo {1.5*\d-\xstep}
\def \yo {-1*\d-\ystep}

\node[draw, rectangle] (u1) at ({\xo+\xstep}, {\yo+\ystep}) {$u_1$};
\node[draw, rectangle] (u2) at ({\xo+2*\xstep}, {\yo+2*\ystep}) {$u_2$};
\draw [fill=black!100, color=black!100] ({\xo+3*\xstep},{\yo+3*\ystep}) circle (0.01cm);
\draw [fill=black!100, color=black!100] ({\xo+2.8*\xstep},{\yo+2.8*\ystep}) circle (0.01cm);
\draw [fill=black!100, color=black!100] ({\xo+3.2*\xstep},{\yo+3.2*\ystep}) circle (0.01cm);
\node[draw, rectangle] (utl) at ({\xo+4*\xstep}, {\yo+4*\ystep}) {$u_{t-1}$};
\node[draw, rectangle] (ut) at ({\xo+5*\xstep}, {\yo+5*\ystep}) {$u_{t}$};
\node[draw, rectangle, dashed] (u't) at ({\xo+5*\xstep-1}, {\yo+5*\ystep}) {$u'_{t}$};
\node[draw, rectangle] (utr) at ({\xo+6*\xstep}, {\yo+6*\ystep}) {$u_{t+1}$};
\draw [fill=black!100, color=black!100] ({\xo+7*\xstep},{\yo+7*\ystep}) circle (0.01cm);
\draw [fill=black!100, color=black!100] ({\xo+6.8*\xstep},{\yo+6.8*\ystep}) circle (0.01cm);
\draw [fill=black!100, color=black!100] ({\xo+7.2*\xstep},{\yo+7.2*\ystep}) circle (0.01cm);
\node[draw, rectangle] (ud) at ({\xo+8*\xstep}, {\yo+8*\ystep}) {$u_{d}$};

\draw [thick]  (w) to (u1);
\draw [thick]  (u1) to (u2);
\draw [thick]  (utl) to (ut);
\draw [thick]  (ut) to (utr);
\draw [thick]  (ut) to (u't);
\draw [thick]  (u2) to ({\xo+2.6*\xstep},{\yo+2.6*\ystep});
\draw [thick]  (utl) to ({\xo+3.4*\xstep},{\yo+3.4*\ystep});
\draw [thick]  (utr) to ({\xo+6.6*\xstep},{\yo+6.6*\ystep});
\draw [thick]  (ud) to ({\xo+7.4*\xstep},{\yo+7.4*\ystep});


\draw [thick, loosely dashed]  (b1v11) to (b2v11);
\draw [thick, loosely dashed]  (b1v11) to (b2v12);
\draw [thick, loosely dashed]  (b1v11) to (b2v1k);

\draw [thick, loosely dashed]  (b1v12) to (b2v11);
\draw [thick, loosely dashed]  (b1v12) to (b2v12);
\draw [thick, loosely dashed]  (b1v12) to (b2v1k);

\draw [thick, loosely dashed]  (b1v1k) to (b2v11);
\draw [thick, loosely dashed]  (b1v1k) to (b2v12);
\draw [thick, loosely dashed]  (b1v1k) to (b2v1k);

\draw [thick, loosely dashed]  (b1v11) to (b2v11);
\draw [thick, loosely dashed]  (b1v11) to (b2v12);
\draw [thick, loosely dashed]  (b1v11) to (b2v1k);
\draw [thick, loosely dashed]  (b1v12) to (b2v11);
\draw [thick, loosely dashed]  (b1v12) to (b2v12);
\draw [thick, loosely dashed]  (b1v12) to (b2v1k);
\draw [thick, loosely dashed]  (b1v1k) to (b2v11);
\draw [thick, loosely dashed]  (b1v1k) to (b2v12);
\draw [thick, loosely dashed]  (b1v1k) to (b2v1k);

\draw [thick, loosely dashed]  (bmv11) to (b1v11);
\draw [thick, loosely dashed]  (bmv11) to (b1v12);
\draw [thick, loosely dashed]  (bmv11) to (b1v1k);
\draw [thick, loosely dashed]  (bmv11) to (b2v11);
\draw [thick, loosely dashed]  (bmv11) to (b2v12);
\draw [thick, loosely dashed]  (bmv11) to (b2v1k);

\draw [thick, loosely dashed]  (bmv12) to (b1v11);
\draw [thick, loosely dashed]  (bmv12) to (b1v12);
\draw [thick, loosely dashed]  (bmv12) to (b1v1k);
\draw [thick, loosely dashed]  (bmv12) to (b2v11);
\draw [thick, loosely dashed]  (bmv12) to (b2v12);
\draw [thick, loosely dashed]  (bmv12) to (b2v1k);

\draw [thick, loosely dashed]  (bmv1k) to (b1v11);
\draw [thick, loosely dashed]  (bmv1k) to (b1v12);
\draw [thick, loosely dashed]  (bmv1k) to (b1v1k);
\draw [thick, loosely dashed]  (bmv1k) to (b2v11);
\draw [thick, loosely dashed]  (bmv1k) to (b2v12);
\draw [thick, loosely dashed]  (bmv1k) to (b2v1k);

\node[below,font=\large\bfseries] at (current bounding box.south) {Figure 2. Graph $L(k,m)$};

\end{tikzpicture}
\end{center}

\begin{theorem}\label{T:lily}
	For $k,m\in \mathbb{N}$, $m\geq 7$ and $k\geq \frac{m-3}{m-6}$, the proportion of good vertices of $L(k,m)$ equals $$\frac{km}{2km+m+2+\Big\lfloor{\frac{\sqrt{8km-48k-8m+33}+1}{2}}\Big\rfloor}.$$

\end{theorem}
\begin{proof}
	Let $G=(V,E)$ the a graph obtained after steps $1-3$ of Construction \ref{C:lily}. Consider a vertex $v=v^1_{11}$ of the graph $G=(V,E)$. Take an arbitrary graph $G'=(V',E')$ such that $V\cap V'=\{u_0\}$. Then, let us define a graph $\tilde{G}=(\tilde{V},\tilde{E})=(V\cup V', E\cup E')$. We do not specify sets $V'$ and $E'$ by now in order to see what conditions on these sets one should impose. 
	
	Our final goal is to find a graph $\tilde{G}$ such that $\Delta_{v}(\tilde{G})=0$. According to Proposition \ref{Prop:1}, we know that $\Delta_v(\tilde{G})=t_{\tilde{G}}(v)+\frac{1}{2}\sum_{u,w \in \tilde{V}\setminus \{v\}}{\delta^v_{\tilde{G}}(u,w)}$. Therefore we proceed by determining $t_{\tilde{G}}(v)$ and $\delta^{v}_{\tilde{G}}(u,w)$ for all vertices $u,w\in \tilde{V}\setminus \{v\}$.
	
	
	By direct calculations we have that $$t_{\tilde{G}}(v)=(1+2)+2(k-1)+3(k-1)+(m-1)(k+2k+3)+t_{G'}(u_0)+|V'|.$$ Therefore, $$t_{\tilde{G}}(v)=3km+2k+3m-5+t_{G'}(u_0)+|V'|.$$  It is easy to see from the construction of the graph, that $\delta^{v}_{\tilde{G}}(u,w)=0$ whenever $u\neq v^1_{21}$ and $w\neq v^1_{21}$. All nonzero values of $\delta^{v}_{\tilde{G}}(u,w)$ are presented in the Table \ref{Tab:lily}.

	\begin{table}[h]
		
		\centering
		\begin{tabular}{ |c|c|c|c|c|}
			
			\hline
			Vertices & $z$ & $v^i_{1j}$ & $v^i_{2j}$ & $v^i$ \\
			\hline
			$v^1_{21}$ & $-2$ & $-2$ & $-2$ & $-2$ \\

			\hline
		\end{tabular}
		\caption{Values of $\delta^{v}_{\tilde{G}}(u,w)$, $z\in V'$, $i\in \{2,3,\dots, m\}$, $j\in \{1,2,\dots, k\}$.}
		{\label{Tab:lily}}
	\end{table}	
Consequently, $$\frac{1}{2}\sum_{u,w \in V\setminus \{v\}}{\delta^v_{\tilde{G}}(u,w)}=-2\Big(|V'|+k(m-1)+k(m-1)+(m-1)\Big)=-2|V'|-4km+4k-2m+2.$$
	
By Proposition \ref{Prop:1} we have that $\Delta_{v}(\tilde{G})=t_{G'}(u_0)-|V'|-km+6k+m-3$. Evidently, $\Delta_{v}(\tilde{G})$ depends on $t_{G'}(u_0)-|V'|$, but not on the structure of a graph $G'$. Therefore, we are free to take arbitrary graph $G'$ such that the following equation holds 
\begin{equation}\label{eq:lily}
t_{G'}(u_0)-|V'|=km-6k-m+3.
\end{equation}

According to steps $4$ and $5$ of Construction \ref{C:lily}, we take a path $(u_0, u_1, u_2, \dots, u_d)$, with one additional vertex $u'_t$ and edge $(u_t,u'_t)$ and finally construct the graph $L(k,m)$, where $$d=\Big\lfloor{\frac{\sqrt{8km-48k-8m+33}+1}{2}}\Big\rfloor ,\, t=km-6k-m+3-\frac{(d+1)(d-2)}{2}.$$ 
In other words, $d$ is a floor of solution to the equation (\ref{eq:lily}) for the path $(u_0, u_1, u_2, \dots, u_d)$: $$(1+2+\dots+d)-(d+1)\leq km-6k-m+3 < (1+2+\dots+d+(d+1))-(d+2).$$ By the assumption of the theorem we know that $m\geq 7$ and $k\geq \frac{m-3}{m-6}$. Hence, $t$ is a non-negative integer and $t<d+1$, so $t$ is defined correctly. Note that if $t=0$ then one does not need to add a vertex $u'_t$ (equation (\ref{eq:lily}) holds). However, in this case we add $u'_0$ as a neighbour of $u_0$ to obtain the resulting graph on the same number of vertices.  
Finally, we conclude that $\Delta_{v_{11}}(L(k,m))=0$.
Based on symmetry of $L(k,m)$, we conclude that for all $z\in \{v^i_{1j}|i\in \{1,2\dots m\}, j\in \{1,2,\dots k \}\}$ the equation $\Delta_{z}(G)=0$ holds.

One may notice that $v_{11}$ and other neighbours of $u_0$ (except $u_1$) are good because its deletion makes a lot of shortest paths longer. It is easy to see from the construction, that other vertices of the graph do not have this property, so either a corresponding transmission is always greater than possible changes of lengths of the shortest paths or the graph with deleted vertex is disconnected. Thus, other vertices can not be good and the theorem is proved.
	
\end{proof}
From the theorem we directly deduce the following
\begin{corollary}\label{Cor:lily}
	For a fixed $m\in \mathbb{N}$, $m\geq 7$, the series of graphs $L(k,m)$, $k\in \mathbb{N}$ has a proportion of good vertices tending to $\frac{1}{2}$ as $k$ tends to infinity.
\end{corollary}

As one could notice from the proof of the last theorem, values of $t$ and $d$ in steps $4$ and $5$ respectively are chosen in order to have $W(L(k,m))-W(L(k,m)-v)=0$. In a similar way for any fixed $z\in \mathbb{Z}$, one can choose $t'$ and $d'$ (denote the resulting graph $W(L'(k,m,z))$) in order to obtain the following equation  

$$W(L'(k,m,z))-W(L'(k,m,z)-v)=z.$$

Since the value $-km+6k+m-3$ for fixed $m\geq 7$ is a decreasing function on $k$ this choice is always possible for $k$ big enough. These arguments allow us to prove the following

\begin{corollary}\label{Cor:lilyz}
	For fixed $m\in \mathbb{N}$, $m\geq 7$, $z\in \mathbb{Z}$, the series of graphs $L'(k,m,z)$, $k\in \mathbb{N}$ has a proportion of vertices $v$, such that $W(L'(k,m,z))-W(L'(k,m,z)-v)=z$, tending to $\frac{1}{2}$ as $k$ tends to infinity.
\end{corollary}

Construction \ref{C:lily} is conceptually close to the one from \cite{BJM2019}. Actually, steps $1$, $2$ and $4$ are similar. The authors also used a number of {\it blocks} (cycles of length at least $7$) having one common vertex and a tree of special structure outgoing from this vertex. Authors were solving a different problem, so they were not interested in analysing and minimizing the size of this tree depending on the number of blocks. 

In fact, if one uses the cycle $C_7$ as a block and a tree of the same structure as it is used in Construction \ref{C:lily} then in every block of the graph there are two vertices that would be good and the whole construction would give a proportion tending to $\frac{1}{3}$ as the number of blocks tends to infinity.

Thus, the better choice of a structure of a block and step $3$ in Construction \ref{C:lily} play a crucial role in Theorem \ref{T:lily}. After numerous attempts to improve presented constructions in order to increase $\alpha$, we tend to think that $\frac{1}{2}$ is the best possible value for constructions of this type.


\section{Conclusion}\label{S:Conclusion} 

In this paper, we consider a problem of constructing graphs $G$ with the given proportion of good vertices, i.e. vertices $v$ such that the following equality takes place $$W(G)=W(G-v).$$
As the main result we built a series of graphs with a proportion tending to $\frac{1}{2}$ with the number of vertices tending to infinity. 

Beside that, we present one intriguing example of a graph with the proportion of good vertices $\frac{2}{3}$.

\begin{example}
	Consider a graph $G$ obtained by adding to the cycle $(v_1,v_2, \dots , v_{12})$ of length $12$ four additional edges: $(v_1,v_3)$, $(v_4,v_6)$, $(v_7,v_9)$, $(v_{10},v_{12})$. This graph on $12$ vertices has $8$ good vertices, namely those of degree $3$.
\end{example}

\begin{center}
\begin{tikzpicture}[scale=0.75]

\def \n {12}
\def \radius {3cm}
\foreach [count=\j] \s in {1,...,\n}
{
	\node[draw, circle] (\j) at ({360/\n * (\s +1)}:\radius) {};
}	

\foreach [remember=\i as \j (initially 1)] \i in {2,...,\n} {
	\draw [thick] (\j) -- (\i);
}
\draw [thick] (1) -- (12);

\draw [thick]  (1) .. controls +(down:1cm) and +(down:1cm) .. (3);
\draw [thick]  (7) .. controls +(up:1cm) and +(up:1cm) .. (9);
\draw [thick]  (4) .. controls +(right:1cm) and +(right:1cm) .. (6);
\draw [thick]  (10) .. controls +(left:1cm) and +(left:1cm) .. (12);

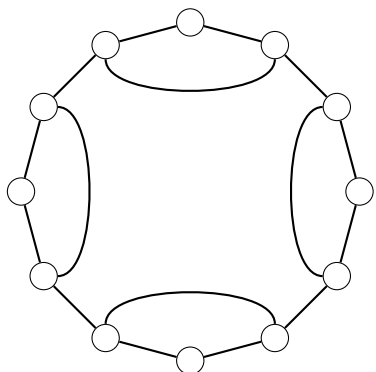
\node[below,font=\large\bfseries] at (current bounding box.south) {Figure 3. Graph with a proportion $\frac{2}{3}$.};
\end{tikzpicture}
\end{center}

Despite the fact that we have not generalized this example for a bigger number of vertices, we expect that there exist an infinite series of graphs with a proportion $\alpha>\frac{1}{2}$, or perhaps even $\alpha$ tending to $1$.

Corollary \ref{Cor:lilyz} shows that instead of solving Problems \ref{Prob:1} and \ref{Prob:2} it is reasonable to consider the following more general problem statements.

\begin{problem}\label{Prob:3}
	For a fixed $z\in \mathbb{Z}$, find all such graphs $G$ that the equality $W(G) - W(G - v) = z$ holds for all their vertices $v$.
\end{problem}

\begin{problem}\label{Prob:4} 
	For a fixed $z\in \mathbb{Z}$ and $\alpha\in (0,1]$, construct an infinite series $S$ of graphs such that for all $G=(V,E)$ from $S$ the following inequality takes place: $$ \frac{|\{v\in V|\Delta_v(G)=z\}|}{|V|} \geq \alpha.$$
\end{problem}

If there are arguments that allow to solve Problems \ref{Prob:1} and \ref{Prob:2} then they would probably work for these ones too. 

\section{Acknowledgements}

The main part on this project was done during the research workshop "Open problems in Combinatorics and Geometry II", held in Adygea in September and October 2020.

The results from Section \ref{S:Bunch} are supported by the Ministry of Science and Higher Education of the Russian Federation in the framework of MegaGrant no 075-15-2019-1926. The work from Section \ref{S:Lily} was carried out within the framework of the state contract of the Sobolev Institute of Mathematics (project no. 0314-2019-0016). 

Authors are grateful to Andrey Dobrynin for interesting discussions on the theme of this paper.


\end{document}